\documentclass[12pt]{amsart}
\usepackage[margin=1.25in]{geometry}
\usepackage{latexsym}
\usepackage{amsmath}
\usepackage{amssymb}
\usepackage{amsthm}
\usepackage{stmaryrd}
\usepackage{amscd}
\usepackage{enumerate} 
\usepackage{amssymb} 
\usepackage{mathrsfs}
\usepackage[all]{xy}
\usepackage{color}
\usepackage{graphicx}
\usepackage{mathtools}
\usepackage{comment}
\usepackage{multirow}
\usepackage[pagebackref]{hyperref}
\hypersetup{colorlinks=true}

\numberwithin{equation}{section}

\title{Steenbrink-type vanishing for surfaces in positive characteristic}
\author{Tatsuro Kawakami}
\email{tatsurokawakami0@gmail.com}
\address{Department of Mathematics, Graduate School of Science, Kyoto University, Kyoto 606-8502, Japan}

\def\phi{\varphi}
\def\epsilon{\varepsilon}
\def\tilde{\widetilde}

\def\log{\operatorname{log}}

\def\Spec{\operatorname{Spec}}

\def\Supp{\operatorname{Supp}}

\def\Exc{\operatorname{Exc}}

\def\m{{\mathfrak m}}

\newcommand{\Q}{\mathbb{Q}} 
\newcommand{\C}{\mathbb{C}} 
 
\newcommand{\Z}{\mathbb{Z}}

\newcommand{\sO}{\mathcal{O}}

\newcommand{\mydot}{{{\,\begin{picture}(1,1)(-1,-2)\circle*{2}\end{picture}\ }}}

\theoremstyle{plain}
\newtheorem{thm}{Theorem}[section] 

\newtheorem{prop}[thm]{Proposition}

\newtheorem{lem}[thm]{Lemma}
\theoremstyle{definition} 
\newtheorem{defn}[thm]{Definition}
\newtheorem{conv}[thm]{Convention}

\theoremstyle{remark}
\newtheorem{rem}[thm]{Remark}

\newtheorem{defn and notation}[thm]{Definition and Notation}

\newtheorem*{clproof}{Proof of Claim}
\newtheorem{clm}{Claim}

\theoremstyle{plain}
\newtheorem{theo}{Theorem}

\keywords{Steenbrink vanishing; singularities; surfaces; positive characteristic}
\subjclass[2020]{14F17,14F10,14J17}

\begin{document}
\tolerance = 9999

\begin{abstract} 
  Let $(X,B)$ be a pair of a normal surface over a perfect field of characteristic $p>0$ and an effective $\Q$-divisor $B$ on $X$.
  We prove that Steenbrink-type vanishing holds for $(X,B)$ if it is log canonical and $p>5$, or it is $F$-pure.  
  We also show that rational surface singularities satisfying the vanishing are $F$-injective.
\end{abstract}

\maketitle
\markboth{Tatsuro Kawakami}{Steenbrink-type vanishing}

\section{Introduction}
The study of vanishing theorems related to differential sheaves is important in the analysis of algebraic varieties. 
Greb--Kebekus--Kov\'{a}cs--Peternell established the following vanishing theorem for log canonical (lc, for short) pairs over the field $\C$ of complex numbers, which can be viewed as a variant of Steenbrink vanishing \cite{Steenbrink}.

\begin{thm}[\textup{\cite[Theorem 14.1]{GKKP}}]\label{thm:GKKP}
 Let $(X,B)$ be an lc pair over $\C$ of $d\coloneqq \dim\,X\geq 2$.
 Let $f\colon Y\to X$ be a log resolution of $(X,B)$ with the reduced $f$-exceptional divisor $E$ and $B_Y\coloneqq f^{-1}_{*}\lfloor B\rfloor+E$, where $f_{*}^{-1}\lfloor B\rfloor$ denotes the proper transform of the round-down $\lfloor B\rfloor$.
 Then we have
 \[
 R^{d-1}f_{*}\Omega^i_Y(\log\,B_Y)(-B_Y)=0
 \]
 for all $i\geq 0$.
\end{thm}

In Theorem \ref{thm:GKKP}, $\Omega^i_Y(\log\,B_Y)$ denotes the sheaf of $i$-th logarithmic diﬀerential forms, and $\Omega^i_Y(\log\,B_Y)(-B_Y)$ denotes $\Omega^i_Y(\log\,B_Y)\otimes \sO_Y(-B_Y)$.
We refer to Subsection \ref{subsection:Notation and terminologies} for details.

The aim of this paper is to prove a positive characteristic analog of Theorem \ref{thm:GKKP} in dimension two. Before stating the main theorem, we briefly recall \textit{$F$-pure singularities}, which can be seen as a positive characteristic analog of lc singularities in characteristic zero. 
Let $(X,B)$ be a pair of a normal variety $X$ over a perfect field of characteristic $p>0$ and an effective $\Q$-divisor $B$.
We say $(X,B)$ is \textit{$F$-pure} if a map 
\[
\sO_{X,x}\xrightarrow{F^e} F^e_{*}\sO_{X,x} \hookrightarrow F^e_{*}\sO_{X}((p^e-1)B)_{x}
\] 
splits as an $\sO_{X,x}$-module homomorphism for all $e>0$ and for every closed point $x\in X$.
Here, the first map is the $e$-times iteration of the Frobenius map $F$, and the second map is obtained by the pushforward of a natural inclusion $\sO_{X,x}\hookrightarrow \sO_{X}((p^e-1)B)_{x}$ by $F^e$.
Hara--Watanabe \cite[Theorem 3.3]{Hara-Watanabe} proved that if $(X,B)$ is $F$-pure and $K_X+B$ is $\Q$-Cartier, then it is lc.
Moreover, it is expected that the modulo $p$-reduction of lc pairs in characteristic zero are $F$-pure for a dense set of primes $p$ \cite[Conjecture 4.8]{Takagi-Watanabe}.

The following theorem states that the vanishing as in Theorem \ref{thm:GKKP} holds when $(X,B)$ is lc and $p>5$, or $(X,B)$ is $F$-pure.

\begin{theo}\label{Introthm:local vanishing for F-pure}
Let $k$ be a perfect field of characteristic $p>0$. 
Let $(X,B)$ be a pair of a normal surface $X$ over $k$ and an effective $\Q$-divisor $B$ on $X$.
Suppose that one of the following holds:
\begin{enumerate}
    \item[\textup{(1)}] $(X,B)$ is lc and $p>5$.
    \item[\textup{(2)}] $(X,B)$ is $F$-pure.
\end{enumerate}
Let $f\colon Y\to X$ be a log resolution with the reduced $f$-exceptional divisor $E$ and $B_Y\coloneqq f^{-1}_{*}\lfloor B\rfloor+E$.
Then we have
 \[
 R^{1}f_{*}\Omega^i_Y(\log\,B_Y)(-B_Y)=0
 \]
for all $i\geq 0$.
\end{theo}
\begin{rem}\label{Introrem}\,
\begin{enumerate}
    \item 
    The crucial statement is the vanishing for $i=1$.
    The case $i=0$ holds for lc surface pairs in all characteristics (Proposition \ref{prop:i=0}).
    The case $i=2$ is nothing but Grauert--Riemenschneider vanishing, which is well-known to hold for normal surfaces in all characteristics.
    \item The vanishing $R^1f_{*}\Omega^1_Y(\log\,B_Y)(-B_Y)=0$ implies the logarithmic extension theorem
    $f_{*}\Omega^1_Y(\log,B_Y)=\Omega^{[1]}_X(\log\,B)$ 
    (see Remark \ref{rem:Vanishing to ext thm} for details).
    Therefore, Theorem \ref{Introthm:local vanishing for F-pure} gives a generalization to Graf's result \cite[Theorem 1.2]{Gra}.
    We remark that our proof is different from Graf's method; we use a splitting of the pullback map of sheaves of differential forms instead of the minimal model program.
\end{enumerate}
\end{rem}

Next, we observe a necessary condition for Theorem \ref{Introthm:local vanishing for F-pure}.
Specifically, in the next theorem, we will see that if $X$ is a normal surface with only rational singularities, the vanishing $R^1f_{*}\Omega^1_Y(\log\,E)(-E)=0$ implies that $X$ is $F$-injective (see Definition \ref{def:F-inj}),
where $f\colon Y\to X$ is a log resolution with the reduced $f$-exceptional divisor $E$.
We recall that $F$-injectivity is a weaker notion than $F$-purity, though they are equivalent when the canonical divisor $K_X$ is Cartier (cf.~\cite[Proposition 3.11]{Takagi-Watanabe}).
Thus, the assumption (2) in Theorem \ref{Introthm:local vanishing for F-pure} is essential.

There exist non-$F$-pure rational double points (RDPs, for short) in each characteristic $p=2$, $3$, and $5$ (see \cite[Table 1]{Kawakami-Takamatsu} for a list of non-$F$-pure RDPs). Therefore, the assumption that $p>5$ in Theorem \ref{Introthm:local vanishing for F-pure} (1) is optimal.

\begin{theo}\label{Introthm:Local vanishing to F-inj}
Let $X$ be a normal surface over a perfect field of positive characteristic and $f\colon Y\to X$ a log resolution with the reduced $f$-exceptional divisor $E$.
Suppose that
\[
R^1f_{*}\Omega^1_Y(\log\,E)(-E)=0.
\]
If $X$ is rational, then $X$ is $F$-injective.
\end{theo}
\begin{rem}\,
\begin{enumerate} 
    \item The affine cone of a supersingular elliptic curve satisfies the above vanishing (Proposition \ref{prop:Gor lc}), but is not $F$-injective. 
    Therefore, the rationality assumption is essential.
    \item We cannot expect $X$ to be $F$-pure. For example, the modulo $p$-reduction of a rational singularity whose graph is of type $(3,3,3)$ is not $F$-pure for $p$ satisfying $p\equiv 2$ modulo $3$ \cite[Theorem 1.2]{Hara(two-dim)}, but satisfies the vanishing for sufficiently large $p$.
    \item 
    We cannot expect $X$ to be $F$-rational. For example, an RDP of type $E_8^1$ in characteristic $5$ is not $F$-rational \cite[Table 1]{Kawakami-Takamatsu}, but satisfies the vanishing since it is $F$-pure (Proposition \ref{prop:canonical}). 
    We refer to \cite[Section 2 and Section 3]{Takagi-Watanabe} for the definition and basic properties of $F$-rationality. 
\end{enumerate}
\end{rem}

For the vanishing for $i=1$ in Theorem \ref{Introthm:local vanishing for F-pure}, we first show that the vanishing can be reduced to finite Galois covers degree prime to $p$ (Proposition \ref{prop:reduction to Galois cover}).
The key ingredient here is a splitting of the pullback map of the sheaves of differential forms (\cite{Kawakami-Totaro},\cite{Totaro(logBott)}).
Then, taking suitable finite Galois covers (which are not necessarily index one covers), Theorem \ref{Introthm:local vanishing for F-pure} can be reduced to the case where $(X,B)$ is toric, or where $X$ is Gorenstein and $B=0$. 
In the second case, the vanishing is obtained by combining the results of Hirokado \cite{Hirokado} and Wahl \cite{Wahl}.
In the first case, we show a variant of Fujino's vanishing \cite[Theorem 5.2]{Fujino} using a splitting of the pullback map of sheaves of logarithmic differential forms again (Proposition \ref{prop:toric cases}).

The proof of Theorem \ref{Introthm:Local vanishing to F-inj} is a simple argument utilizing Cartier operators (Subsection \ref{sebsec:Hara Cartier operator}), and does not require explicit calculations or classifications of singularities.

Finally, we consider another variant of Steenbrink vanishing: 

\begin{thm}[\textup{\cite[Theorem 1.9]{KS21}}]\label{thm:Kebekus-Schnell}
    Let $X$ be a variety over $\C$ of $\dim\,d\geq 2$. 
    Let $f\colon Y\to X$ be a log resolution with the reduced $f$-exceptional divisor $E$.
    If $R^{d-1}f_{*}\sO_Y=0$, then we have
    \[
    R^{d-1}f_{*}\Omega_Y^{1}(\log\,E)=0.
    \]
\end{thm}
\begin{rem}
    Even if $X=\mathbb{A}_{\C}^2$ and $f$ is the blow-up at origin,  we have
    \[
    R^{1}f_{*}\Omega^{2}_X(\log\,E)\neq 0,
    \]
    as proven in \cite[Appendix, Remark 1.4]{Mustata-Popa}.
\end{rem}

The vanishing theorem mentioned above was initially conjectured by Musta\c{t}\u{a}--Olano--Popa \cite{Mustata-Olano-Popa} for rational singularities over $\C$ and was solved in isolated singularity and toric cases. The complete solution of this conjecture was provided by Kebekus--Schnell \cite[Theorem 1.9]{KS21}, as stated in Theorem \ref{thm:Kebekus-Schnell}.

Unfortunately, the above vanishing hardly holds in positive characteristic.
Indeed, we can observe that the vanishing can fail even for toric surfaces;
the above vanishing implies the surjectivity of the restriction map $f_{*}\Omega^1_Y\hookrightarrow \Omega^{[1]}_X$ (Remark \ref{rem:Vanishing to ext thm}), but it is known that this surjectivity is not always true for toric surfaces in positive characteristic (\cite[Proposition 1.5]{Langer19} or \cite[Example 10.2]{Gra}). 
As a corollary of Theorem \ref{Introthm:local vanishing for F-pure}, we obtain some sufficient conditions for the vanishing:

\begin{theo}\label{Introthm:tame quotient}
Let $X$ be a normal surface over an algebraically closed field of characteristic $p>0$ and 
$f\colon Y\to X$ a log resolution with the reduced $f$-exceptional divisor $E$.
Suppose that one of the following holds:
\begin{enumerate}
    \item[\textup{(1)}] $X$ has only quotient singularities by finite group schemes order prime to $p$.
    \item[\textup{(2)}] $X$ is rational, lc, $p>5$, and the determinant of the intersection matrix of $E$ is prime to $p$.
    \item[\textup{(3)}] $X$ is rational, $F$-pure, and the determinant of the intersection matrix of $E$ is prime to $p$.
\end{enumerate}
Then we have
$R^{1}f_{*}\Omega^1_Y(\log\,E)=0$.
\end{theo}

\section{Preliminaries}

\subsection{Notation and terminology}\label{subsection:Notation and terminologies}
Throughout the paper, we work over a fixed algebraically field $k$ of characteristic $p>0$ unless stated otherwise.
\begin{enumerate}
    \item A \textit{variety} means an integral separated scheme of finite type.
    \item For a proper birational morphism $f\colon Y\to X$ of schemes, we denote by $\Exc(f)$ the reduced $f$-exceptional divisor.
    \item For a $\Q$-divisor $B$ with the irreducible decomposition $B=\sum_i b_iB_i$ , we denote by $\lfloor B \rfloor$ the \textit{round-down} $\sum_i \lfloor b_i \rfloor B_i$.
    \item A pair $(X,B)$ consists of a normal variety $X$ and an effective $\Q$-divisor $B$. A \textit{pointed pair} $(x\in X,B)$ consists of a pair $(X,B) $ and a closed point $x\in X$. When $B=0$, we simply write $(x\in X,B)$ as $(x\in X)$.
    \item A morphism $g\colon (x'\in X',B')\to (x\in X,B)$ of pointed pairs is a morphism $g\colon X'\to X$ of normal varieties such that $g(B')\subset B$ and $g(x')=x$.
    \item A morphism $\phi\colon (u\in U,B_U)\to (x\in X,B)$ of pointed pairs is said to be
    an \textit{\'etale neighborhood of $x$} if $\phi\colon U\to X$ is \'etale and $B_U$ is isomorphic to the scheme-theoretic inverse image $\phi^{-1}B$.
    \item A singularity $(x\in X,B)$ is an equivalent class of pointed pairs under common \'etale neighborhoods.
    \item Given a smooth variety $X$, a reduced simple normal crossing (snc, for short) divisor $B$, and an integer $i\geq 0$, we denote by $\Omega_X^{i}(\log\,B)$ the sheaf of logarithmic differential forms.
    For a $\Q$-divisor $D$ on $X$, we denote $\Omega^i_X(\log\,B)\otimes \sO_X(\lfloor D\rfloor)$ by $\Omega^i_X(\log\,B)(D)$.
    \item Given a normal variety $X$, a reduced divisor $B$ on $X$, a $\Q$-divisor $D$ on $X$, and an integer $i\geq 0$, we denote $j_{*}(\Omega_U^{i}(\log\,B)(\lfloor D\rfloor))$ by $\Omega_X^{[i]}(\log B)(D)$, where $j\colon U\hookrightarrow X$ is the inclusion of the snc locus $U$ of $(X,B)$.
    \item For the definition of the singularities of pairs appearing in the MMP (such as \emph{canonical, klt, plt, lc}) we refer to \cite[Definition 2.8]{Kol13}. Note that we always assume that the boundary divisor is effective although \cite[Definition 2.8]{Kol13} does not impose this assumption. 
\end{enumerate}

\subsection{\texorpdfstring{$F$}--pure singularities}
In this subsection, we gather facts about $F$-purity.

\begin{defn}\label{def:F-pure}
Let $(X,B)$ be a pair of a normal variety $X$ over a perfect field of characteristic $p>0$ and an effective $\Q$-divisor $B$. 
We say $(X,B)$ is \textit{$F$-pure at a closed point $x\in X$} if 
\[
\phi\colon\sO_{X,x}\to F^e_{*}\sO_{X,x} \hookrightarrow F^e_{*}\sO_{X}((p^e-1)B)_{x}
\] 
splits as an $\sO_{X,x}$-module homomorphism for all $e>0$. Here, the first map is the $e$-times iteration of the Frobenius map, and the second map is obtained by the pushforward of a natural inclusion $\sO_{X,x}\hookrightarrow\sO_{X}((p^e-1)B)_{x}$ by $F^e$.
We say that $(X,B)$ is \textit{$F$-pure} if it is $F$-pure at every closed point.
\end{defn}
\begin{rem}\label{rem:base change of F-purity}
    We take $(X,B)$ as in Definition \ref{def:F-pure}. Let $\overline{k}$ be an algebraic closure of $k$.
    Then $(X,B)$ is $F$-pure if and only if $(X\times_{k} \overline{k}, B\times_{k} \overline{k})$ is $F$-pure, as follows.
    A splitting of the map
    \[
    \phi\colon \sO_{X,x}\to F^e_{*}\sO_{X}((p^e-1)B)_{x}
    \] 
    is equivalent to the surjectivity of the evaluation map
    \[
    \Phi\colon \mathcal{H}\mathit{om}_{\sO_{X,x}}(F^e_{*}\sO_{X}((p^e-1)B)_{x}, \sO_{X,x})\to 
    \mathcal{H}\mathit{om}_{\sO_{X,x}}(\sO_{X,x}, \sO_{X,x})\cong \sO_{X,x}, 
    \]
    which sends $\psi$ to $\psi\circ \phi$.
    Since $\Phi$ is surjective if and only if $\Phi\otimes_{k}\overline{k}$ is surjective, we conclude that $(X,B)$ is $F$-pure if and only if $(X\times_{k} \overline{k}, B\times_{k} \overline{k})$ is $F$-pure.
\end{rem}
\begin{rem}\label{rem:reduction and F-purity}
    We take $(X,B)$ as in Definition \ref{def:F-pure}. Let $B'$ be an effective $\Q$-divisor with $0\leq B'\leq B$.
    If $(X,B)$ is $F$-pure, then $(X,B')$ is $F$-pure.
    This follows from the fact that the map
    \[
    \phi\colon \sO_{X,x}\to F^e_{*}\sO_{X}((p^e-1)B)_{x}
    \] 
    factors through 
    \[
    \phi'\colon \sO_{X,x}\to F^e_{*}\sO_{X}((p^e-1)B')_{x}
    \]
    for every closed point $x\in X$.
\end{rem}
\begin{rem}
    We say a singularity $(x\in X,B)$ of a pair of a normal variety $X$ and an effective $\Q$-divisor $B$ is \textit{$F$-pure} if
    some representative $(X,B)$ is $F$-pure at $x$.
    In this case, we can confirm that every representative $(X,B)$ is $F$-pure at $x$ since the map
    \[
    \phi\colon\sO_{X,x}\to F^e_{*}\sO_{X}((p^e-1)B)_{x}
    \] 
    splits if and only if the map
    \[
    \phi\otimes_{\sO_{X,x}}\sO_{X,x}^{\wedge}\colon\sO_{X,x}^{\wedge}\to F^e_{*}\sO_{X}((p^e-1)B)_x^{\wedge}
    \] 
    splits, where $(-)_{x}^{\wedge}$ denotes the completion at $x\in X$.
\end{rem}
\begin{rem}\label{rem:F-pure-to-lc}
    If a pair $(X,B)$ of a normal surface $X$ and a reduced divisor $B$ is $F$-pure, then it is lc ($K_X+B$ is automatically $\Q$-Cartier) by \cite[Proposition 2.2 (b)]{BBKW}. Note that the $F$-purity is equivalent to the sharply $F$-purity when $B$ is reduced, and thus we can apply \cite[Proposition 2.2 (b)]{BBKW}.
\end{rem}

\section{Local vanishing theorems}

\subsection{Independence on log resolutions}

In this paper, we discuss the following vanishings:

\begin{defn}\label{def:local vanishings}
    Let $(X,B)$ be a pair of a normal surface $X$ and a reduced divisor $B$.
    We say $(X,B)$ satisfies \textit{Vanishing (I-i) at $x\in X$}
    if  \[R^1f_{*}\Omega^i_Y(\log\,B_Y)(-B_Y)_x=0\]
    holds for every log resolution $f\colon Y\to X$ of $(X,B)$ with $E\coloneqq\Exc(f)$ and $B_Y\coloneqq f^{-1}_{*}B+E$.

    We say $(X,B)$ satisfies \textit{Vanishing (II) at $x\in X$}
    if  \[R^1f_{*}\Omega^1_Y(\log\,B_Y)_x=0\]
    holds for every log resolution $f\colon Y\to X$ of $(X,B)$.
    
    We say $(X,B)$ satisfies \textit{Vanishing (I-i) (resp.~(II))} if it satisfies Vanishing (I-i) (resp.~(II)) at every closed point $x\in X$.
\end{defn}

\begin{rem}\label{rem:vanishing for singularities}
    For a singularity $(x\in X,B)$, if Vanishing (I-i) (resp.~(II)) holds at $x$ for some representative, 
    then it holds for every representative. For example, for Vanishings (I-i), this fact follows from the fact that
    \[R^1f_{*}\Omega^i_Y(\log\,B_Y)(-B_Y)_x=0\]
    if and only if 
    \[R^1f_{*}\Omega^i_Y(\log\,B_Y)(-B_Y)^{\wedge}_x=0.\]
\end{rem}

In the rest of this subsection, we confirm that Vanishings (I-i) and (II) do not depend on the choice of log resolutions.

We first show a variant of Fujino's vanishing \cite[Theorem 5.2]{Fujino} using a splitting of the pullback map of sheaves of logarithmic differential forms by Totaro \cite[Lemma 2.1]{Totaro(logBott)}. 

\begin{prop}\label{prop:toric cases}
    Let $(X,B)$ be a toric pair such that $B$ is reduced.
    Let $f\colon Y\to X$ be a toric log resolution of $(X,B)$ with $E\coloneqq \Exc(f)$ and $B_Y\coloneqq f^{-1}_{*}B+E$.
    Then $R^jf_{*}\Omega^i_Y(\log\,B_Y)(-B_Y)=0$ for all $i\geq 0$ and $j>0$.
\end{prop}
\begin{proof}
    We fix a positive integer $l$ that is not divisible by $p$.
    We have an $l$ times multiplication map $g\colon Y\to Y$ \cite[Section 2.1]{Fujino}.
    The ramification index of every Cartier divisor on $Y$ with respect to $g^n$ is equal to $l^n$, and thus
    we have a split injection
    \begin{equation}\label{eq:split1}
        \Omega^i_Y(\log\,B_Y)(-B_Y)\to (g^{n})_{*}\Omega^i_Y(\log\,B_Y)(-B_Y) 
    \end{equation}
    for every $n>0$ by \cite[Lemma 2.1]{Totaro(logBott)}.
    We note that \eqref{eq:split1} factors into
    \begin{multline*}
        \Omega^i_Y(\log\,B_Y)(-B_Y)=\Omega^i_Y(\log\,B_Y)\otimes \sO_Y(-B_Y)\\\xrightarrow{(g^n)^{*}\otimes \sO_Y(-B_Y)} ((g^{n})_{*}\Omega^i_Y(\log\,B_Y))\otimes \sO_Y(-B_Y)=(g^{n})_{*}\Omega^i_Y(\log\,B_Y)(-l^nB_Y)
    \end{multline*}
    and a natural inclusion
    \begin{equation}\label{eq:split3}
        (g^{n})_{*}\Omega^i_Y(\log\,B_Y)(-l^nB_Y)\to (g^{n})_{*}\Omega^i_Y(\log\,B_Y)(-B_Y)
    \end{equation}
    Take an $f$-ample $\Q$-divisor $A$ such that $\lfloor A\rfloor=-E$ and 
    fix $m\gg0$ such that
    \[
    R^jf_{*}\Omega^i_Y(\log\,B_Y)(-f^{-1}_{*}B+mA)=0
    \]
    for all $j>0$.
    For $n\gg0$, the map \eqref{eq:split3} factors into
    \begin{align*}
        (g^{n})_{*}\Omega^i_Y(\log\,B_Y)(-l^nB_Y)&\hookrightarrow (g^n)_{*}\Omega^i_Y(\log\,B_Y)(-f^{-1}_{*}B+mA)\\ 
        &\hookrightarrow (g^{n})_{*}\Omega^i_Y(\log\,B_Y)(-B_Y)
    \end{align*}
    of natural inclusions.
    Thus, \eqref{eq:split1} factors through
    \begin{equation}\label{eq:split2}
        \Omega^i_Y(\log\,B_Y)(-B_Y)\to (g^n)_{*}\Omega^i_Y(\log\,B_Y)(-f^{-1}_{*}B+mA),
    \end{equation}
    and \eqref{eq:split2} also splits.
    Applying $R^jf_{*}$ to \eqref{eq:split2}, we obtain 
    \[
    R^jf_{*}\Omega^i_Y(\log\,B_Y)(-B_Y) \hookrightarrow (g^n)_{*}R^jf_{*}\Omega^i_Y(\log\,B_Y)(-f^{-1}_{*}B+mA)=0, 
    \]
    as desired.
\end{proof}

\begin{lem}\label{lem:smooth cases}
    Let $(X,B)$ be a pair of a smooth surface $X$ and a reduced snc divisor $B$. 
    We fix $x\in X$ and 
    let $f\colon Y\to X$ be the blow-up at $x\in X$.
    Set $E\coloneqq\Exc(f)$ and $B_Y\coloneqq f_{*}^{-1}B+E$. 
    Then the following holds:
    \begin{enumerate}
        \item[\textup{(1)}] $R^jf_{*}\Omega^i_Y(\log\,B_Y)(-B_Y)=0$ for all $i\geq 0$ and all $j>0$. 
        \item[\textup{(2)}] $f_{*}\sO_Y(-B_Y)=\sO_X(-B)$ if $x\in B$.
        \item[\textup{(3)}] $f_{*}\Omega^i_Y(\log\,B_Y)(-B_Y)=\Omega^i_X(\log\,B)(-B)$ for $i>0$.
    \end{enumerate}
\end{lem}
\begin{proof}
    Since $B$ is snc, taking an \'etale neighborhood of $x\in X$, we may assume that $(X,B)$ and $f$ are toric.
    Then (1) follows from Proposition \ref{prop:toric cases}.
    
    We next show (2). 
    Since $x \in B$, we have $f^{*}B=f_{*}^{-1}B+iE$ for $i>0$, and thus $-B_Y=f_{*}^{-1}(-B)-E=f^{*}(-B)+(i-1)E$. Thus, the projection formula shows 
    \[
    f_{*}\sO_Y(-B_Y)=f_{*}\sO_Y(f^{*}(-B)+(i-1)E)=\sO_X(-B),
    \]
    as desired.

    Finally, we show (3).
    We have $f_{*}\omega_Y=\omega_X$ since $X$ is smooth, and the case $i=2$ holds.
    Suppose that $i=1$.
    We have an exact sequence (cf.~\cite[Lemma 4.1]{Mustata-Olano-Popa})
    \[
    0\to \Omega^1_Y(\log\,B_Y)(-B_Y) \to \Omega^1_Y \to \Omega^1_{E}\oplus (\oplus_i \Omega^1_{B'_i})\to 0,
    \]
    where $B=\sum_i B_i$ is the irreducible decomposition and $B'_i\coloneqq f_{*}^{-1}B_i$.
    Taking the pushforward by $f$, we obtain
    \[
        0\to f_{*}\Omega^1_Y(\log\,B_Y)(-B_Y) \to \Omega^1_X \to \oplus_i \Omega^1_{B_i}
    \]
    Comparing with the exact sequence
     \[
    0\to \Omega^1_X(\log\,B)(-B) \to \Omega^1_X \to \oplus_i \Omega^1_{B_i}\to 0, 
    \]
    we obtain 
    \[
     f_{*}\Omega^1_Y(\log\,B_Y)(-B_Y)=\Omega_X^1(\log\,B)(-B),
    \]
    as desired.
\end{proof}

\begin{prop}\label{prop:independence}
    The vanishing (I-i) and (II) do not depend on the choice of log resolutions for all $i \geq 0$.
\end{prop}
\begin{proof}
    We discuss only the case of Vanishing (I-i) since the other is similar.
    Let $(X,B)$ be a pair of a normal surface $X$ and a reduced divisor $B$.
    We take two log resolutions $f\colon Y\to X$ and $f'\colon Y'\to X$ of $(X,B)$.
    Set $E\coloneqq\Exc(f)$, $E'\coloneqq \Exc(f')$, $B_Y\coloneqq f^{-1}_{*}B+E$, and $B_{Y'}\coloneqq (f')^{-1}_{*}B+E'$.
    We aim to show that 
    \begin{equation}\label{eq:independent}
        R^1f_{*}\Omega^i_{Y}(\log\,B_{Y})(-B_{Y})\cong R^1(f')_{*}\Omega^i_{Y'}(\log\,B_{Y'})(-B_{Y'}).
    \end{equation}
    Since we can find a log resolution that dominates both $Y$ and $Y'$, we may assume that $f'$ factors through $f$.
    Let $h\colon Y'\to Y$ be the induced morphism.
    Since $Y$ and $Y'$ are smooth, we have a decomposition $h=h_n\circ h_{n-1} \circ \cdots \circ h_1$, where each $h_i$ is a blow-up at a closed point.
    We fix $x\in X$ and prove the equality of \eqref{eq:independent} at $x$.
    By replacing $X$ with an open neighborhood of $x$, we may assume that $f'$ is a log resolution at $x$, i.e., that $Y'\setminus E'\cong X\setminus \{x\}$.
    Then each $h_j$ is a blow-up at a closed point that is contained in $\Exc(f\circ h_n\circ\cdots \circ h_{j+1})$, and thus we can apply Lemma \ref{lem:smooth cases} (2) to $h_j$ if $j<n$.
    Therefore, a repeated use of Lemma \ref{lem:smooth cases} and the Leray spectral sequence gives
    \[
    R^1h_{*}\Omega^i_{Y'}(\log\,B_{Y'})(-B_{Y'})=0,
    \]
    and we obtain an isomorphism
    \[
    R^1f_{*}h_{*}\Omega^i_{Y'}(\log\,B_{Y'})(-B_{Y'})\cong R^1(f')_{*}\Omega^i_{Y'}(\log\,B_{Y'})(-B_{Y'}).
    \]
    When $f$ is not an isomorphism, we have $B_Y\neq 0$ and 
    \[
    h_{*}\Omega^i_{Y'}(\log\,B_{Y'})(-B_{Y'})\cong \Omega_Y^i(\log\,B_Y)(-B_Y).
    \]
    by Lemma \ref{lem:smooth cases} (2) and (3).
    Therefore, we obtain \eqref{eq:independent}.
    When $f$ is an isomorphism, then both hand sides of \eqref{eq:independent} are zero, and in particular \eqref{eq:independent} is still true.
    Thus, we conclude.
\end{proof}

\begin{rem}\label{rem:Vanishing to ext thm}
Let $(x\in X,B)$ be a singularity of a pair of a normal surface $X$ and a reduced divisor $B$.
Let $f\colon Y\to X$ be a log resolution of $(X,B)$ at $x\in X$ with $E\coloneqq \Exc(f)$ and $B_Y\coloneqq f_{*}^{-1}B+E$.
    \begin{enumerate}
        \item 
        By the formal duality, we have
        \[
        R^1f_{*}\Omega^1_Y(\log\,B_Y)(-B_Y)\cong H^1_{E}(\Omega^1_Y(\log\,B_Y)).
        \]
        We have an exact sequence
        \[
        0\to H^0(Y, \Omega^1_Y(\log\,B_Y)) \to H^0(Y\setminus E, \Omega^1_Y(\log\,B_Y)) \to H^1_{E}(\Omega^1_Y(\log\,B_Y)).
        \]
        Therefore, Vanishing (I-1) implies the logarithmic extension theorem, i.e., the surjectivity of the restriction map
        \[
        f_{*}\Omega_Y^1(\log\,B_Y) \hookrightarrow \Omega^{[1]}_X(\log\,B).
        \]
        \item Similarly, Vanishing (II) implies the surjectivity of the restriction map
        \[
        f_{*}\Omega_Y^1(\log\,B_Y)(-B_Y) \hookrightarrow \Omega^{[1]}_X(\log\,B)(-B).
        \]
        In particular, if $B=0$, then 
        \[
        f_{*}\Omega_Y^1(\log\,E)(-E) (\hookrightarrow f_{*}\Omega_Y^1) \hookrightarrow \Omega^{[1]}_X
        \]
        is surjective.
    \end{enumerate}
\end{rem}

\subsection{Reduction to finite Galois covers}

In this subsection, we observe that Vanishing (I-i) and (II) can be reduced to finite Galois covers degree prime to $p$.

\begin{defn}
    A finite surjective morphism $g\colon X'\to X$ of normal varieties is called a \textit{finite Galois cover} if the extension $K(X')/K(X)$ of function fields is Galois.
\end{defn}

\begin{lem}\label{lem:tame ramification}
Let $g\colon X'\to X$ be a finite Galois cover of degree prime to $p$ of normal varieties.
Then the ramification index of every divisor on $X'$ is not divisible by $p$.
\end{lem}
\begin{proof}
By shrinking $X$, we may assume that $X$ and $X'$ are affine.
Let $G$ be a Galois group corresponds to $K(X')/K(X)$.
Then $G$ acts on $X'$ since $X'$ is normal.
Moreover, the induced map $X'/G\to X$ is an isomorphism since it is integral and birational, and both $X$ and $X'/G$ are normal.

Let $B'$ be a prime divisor on $X'$ and $B\coloneqq g_{*}B'$.
Let $\Supp(g^{*}B)=\sum_{i=1}^{n}B'_i$ and $B'_1\coloneqq B'$.
Let $r_i$ be the ramification index of $g$ at $B'_i$.
Then we have \[\deg(g)B=g_{*}g^{*}B=g_{*}(\sum_{i=1}^{n}r_iB'_i)=\sum_{i=1}^{n}r_ie_iB\] for some $e_i\in\Z_{>0}$.
Since $G$ acts on the set $\{B'_1,\ldots,B'_n\}$ transitively, we have $r_1=\cdots =r_n=:r$ and $e_1=\cdots =e_n=:e$.
Therefore we obtain $\deg(g)=nre$, which shows that $r$ is not divisible by $p$.
\end{proof}

\begin{prop}\label{prop:reduction to Galois cover}
    Let $g\colon X'\to X$ be a finite Galois cover degree prime to $p$ of normal surfaces.
    Let $B$ be a reduced divisor on $X$ and $B'\coloneqq \Supp(g^{*}B)$.
    If $(X',B')$ satisfies Vanishing (I-i) for some $i\geq 0$ (resp.~Vanishing (II)), then so does $(X,B)$.
\end{prop}
\begin{proof}
    We discuss only the case of Vanishing (I-i) since the other is similar.
    Let $f\colon Y\to X$ be a log resolution of $(X,B)$ with $E\coloneqq \Exc(f)$ and $B_Y\coloneqq f_{*}^{-1}B+E$.
    We aim to prove that
    \[
    R^1f_{*}\Omega^i_Y(\log\,B_Y)(-B_Y)=0.
    \]
    Let $Y'$ be the normalization of the component of the fiber product dominating $X'$, and let $f'\colon Y'\to X'$ and $g_Y\colon Y'\to Y$ be compositions of the normalization and projections.
    Then $g_Y\colon Y'\to Y$ is a finite Galois cover degree prime to $p$.
    Let $B_{Y'}\coloneqq (f')_{*}^{-1}B'+\Exc(f')=\Supp(g_Y^{*}B_Y)$.
   Let $h\colon \tilde{Y} \to Y$ be a log resolution of $(Y',B_{Y'})$ 
   with $B_{\tilde{Y}}\coloneqq h^{-1}_{*}B_{Y'}+\Exc(h)$,
   and $\tilde{f}\coloneqq f'\circ h$.
   We now have the following commutative diagram:
    \begin{equation*}
    \xymatrix@C=80pt{
    \tilde{Y}\ar[r]^-{h}\ar[rd]_-{\tilde{f}}& Y'\ar[r]^{g_Y} \ar[d]^{f'} & Y\ar[d]^-{f}\\ 
      &  X'\ar[r]^-{g}   & X.}
\end{equation*}
    Since $\tilde{f}\colon \tilde{Y}\to X'$ be a log resolution of $(X',B')$ and $B_{\tilde{Y}}=\tilde{f}_{*}^{-1}B'+\Exc(\tilde{f})$, we have 
    \[
    R^1\tilde{f}_{*}\Omega^i_{\tilde{Y}}(\log\,B_{\tilde{Y}})(-B_{\tilde{Y}})=0
    \] 
    by assumption. 
    By the spectral sequence, we have an injective map
    \[
    R^1(f')_{*}h_{*}\Omega^i_{\tilde{Y}}(\log\,B_{\tilde{Y}})(-B_{\tilde{Y}})\hookrightarrow R^1\tilde{f}_{*}\Omega^i_{\tilde{Y}}(\log\,B_{\tilde{Y}})(-B_{\tilde{Y}})=0.
    \]
    We define a coherent sheaf $\mathcal{C}$ as a cokernel of the natural inclusion 
    \[
    h_{*}\Omega^i_{\tilde{Y}}(\log\,B_{\tilde{Y}})(-B_{\tilde{Y}})\hookrightarrow \Omega^{[i]}_{Y'}(\log\,B_{Y'})(-B_{Y'}).
    \]
    Then $\Supp(\mathcal{C})\subset h(\Exc(h))$, and in particular, $\dim\,\Supp(\mathcal{C})\leq 0$.
    Thus, we have 
    \[
    R^{1}(f')_{*}\Omega^{[i]}_{Y'}(\log\,B_{Y'})(-B_{Y'})=0.
    \]
    By Lemma \ref{lem:tame ramification} and \cite[Lemma 2.1]{Totaro(logBott)}, the pullback map 
    \[
    \Omega_Y^i(\log\,B_Y)(-B_Y)\hookrightarrow (g_{Y})_{*}\Omega^{[i]}_{Y'}(\log\,B_{Y'})(-B_{Y'})
    \] 
    splits.
    Taking $R^{1}f_{*}$, we have a split injection
    \[
    R^{1}f_{*}\Omega^i_Y(\log\,B_Y)(-B_Y)\hookrightarrow g_{*}R^1(f')_{*}\Omega^{[i]}_{Y'}(\log\,B_{Y'})(-B_{Y'})=0,
    \]
    and thus we conclude.
\end{proof}

\subsection{Reduction to the toric or Gorenstein case.}

In the previous subsection, we observed that Vanishing (I-i) and (II) can be reduced to finite Galois covers degree prime to $p$.
In this subsection, we apply this fact to show that Theorem \ref{Introthm:local vanishing for F-pure} can be reduced to the case where $(X,B)$ is toric or where $X$ is Gorenstein and $B=0$.
From now, we freely use the classification of lc surface singularities in \cite[Section 3.3]{Kol13}. The list in \cite[Section 7.B]{Gra} is also useful.

\begin{defn}
Let $G$ be an affine group scheme of finite type.
We say that $G$ is a \textit{linearly reductive group scheme} if every linear representation of $G$ is semi-simple.
\end{defn}

\begin{defn}\label{def:quotient singularity}
Let $G$ a finite group scheme.
A singularity $(x\in X)$ of a variety $X$ is said to be a \textit{quotient singularity by $G$} 
if there exists a faithful action of $G$ on $\Spec k\llbracket x_1,\ldots,x_d\rrbracket$ fixing the closed point such that $\sO_{X,x}^{\wedge}\cong k\llbracket x_1,\ldots,x_d\rrbracket^{G}$.
\end{defn}

\begin{defn}
We say that a singularity $(x\in X,B)$ such that $X$ is a normal variety and $B$ is a reduced divisor is \textit{toric} if there exists 
some representative $(x\in X,B)$ such that $(X,B)$ is a toric pair.
\end{defn}

\begin{lem}\label{lem:quotient sing is toric up to Galois}
    Let $(x\in X)$ be a quotient singularity by a finite linearly reductive group scheme $G$.
    Then there exists a finite Galois cover $g\colon (x'\in X')\to (x\in X)$ of singularities of degree prime to $p$ such that $(x'\in X')$ is a toric singularity.
\end{lem}
\begin{proof}
   The assertion follows from the proof of \cite[Theorem 2.12]{Kaw4}, but
   we provide a sketch of proof for the convenience of the reader.
   By the Artin approximation, we can take $(o \in \mathbb{A}_k^d/G)$ as a representative of a singularity $(x\in X)$.
   Let $G^{\circ}$ be the connected component containing the identity, which is a normal subgroup scheme of $G$.
   Then $G^{\circ}$ is diagonalizable and $G/G^{\circ}$ is finite order prime to $p$ (see \cite[Section 2]{Hashimoto}).
   Since the natural torus action on $\mathbb{A}_k^d$ descends to $\mathbb{A}_k^d/G^{\circ}$, the quotient $\mathbb{A}_k^d/G^{\circ}$ is toric.
   Therefore, the induced map $g\colon (o\in \mathbb{A}_k^d/G^{\circ})\to (x\in X)$ is the desired one.
\end{proof}

\begin{lem}\label{lem:F-pure klt is quotient sing}
    Let $(x\in X)$ be a normal surface singularity.
    If $(x\in X)$ is $F$-pure, klt and not canonical, then $(x \in X)$ is a quotient singularity by a finite linearly reductive group scheme.
\end{lem}
\begin{proof}
    By \cite[Theorem 1.2]{Hara(two-dim)}, it follows that $(x\in X)$ is $F$-regular.
    Then the assertion follows from \cite[Theorem 1.6]{LMM}.
\end{proof}

Let $(x\in X,B)$ be a singularity such that $K_X+B$ is $\Q$-Cartier of index $m$ prime to $p$.
Then we can construct a $\mu_m$-cover $g\colon (x'\in X',B')\to (x\in X,B)$ of singularities such that $K_{X'}+B'$ is Cartier and $g$ is \'etale in codimension one, which is called the \textit{index one cover of $(x\in X,B)$}. We refer to \cite[Definition 2.49]{Kol13} for details.

\begin{lem}\label{lem:reduction of F-pure dlt case}
    Let $(x\in X,B)$ be a surface singularity such that $B$ is reduced.
    If $(x\in X,B)$ is $F$-pure and plt,
    then one of the following holds:
    \begin{enumerate}
        \item[\textup{(1)}] $B=0$ and $(x\in X)$ is $F$-pure canonical.
        \item[\textup{(2)}] There exists a finite Galois cover $g\colon (x'\in X',B')\to (x\in X,B)$ of degree prime to $p$ such that $(x'\in X',B')$ is a toric singularity, where $B'\coloneqq \Supp(g^{*}B)$.
    \end{enumerate}
\end{lem}
\begin{proof}
    Suppose that $B=0$ .
    If $X$ is Gorenstein, then (1) holds. If $X$ is not Gorenstein, then (2) holds by Lemmas \ref{lem:F-pure klt is quotient sing} and \ref{lem:quotient sing is toric up to Galois}.   
    If $B\neq 0$, then $(x\in X,B)$ is a plt cyclic quotient singularity \cite[3.35 (1)]{Kol13}, which is a toric singularity by \cite[Theorem 2.13]{Lee-Nakayama}. Thus, (2) holds. 
    We note that, in our situation, $B_2=0$ in the notation of \cite[Theorem 2.13]{Lee-Nakayama}, but we can reduce the $B_2\neq 0$ case as in \cite[Proof of Theorems 2.11 and 2.13]{Lee-Nakayama}, and thus can apply \cite[Theorem 2.13]{Lee-Nakayama}.
\end{proof}

\begin{rem}\label{rem:reduction of dlt case}
    In Lemma \ref{lem:reduction of F-pure dlt case}, we can replace the assumption that $(X,B)$ is $F$-pure by that $p>5$
    since every plt surface pair is $F$-pure if $p>5$ (\cite[Theorem 1.1]{Hara(two-dim)} and \cite[Theorem 4.5 (1)]{Hara-Watanabe}).
\end{rem}

\begin{lem}\label{lem:reduction of lc case}
    Let $(x\in X,B)$ be an lc surface singularity such that $B$ is reduced.
    Suppose that $(x\in X,B)$ is not plt and $p>5$.
    Then one of the followings holds:
    \begin{enumerate}
        \item[\textup{(1)}] $B=0$ and there exists a finite Galois cover $g\colon (x'\in X')\to (x\in X)$ of degree prime to $p$ such that $X'$ is Gorenstein lc.
        \item[\textup{(2)}] $B\neq 0$ and there exists a finite Galois cover $g\colon (x'\in X',B)\to (x\in X, B)$ of degree prime to $p$ such that $(x'\in X',B)$ is a toric singularity.
    \end{enumerate}
\end{lem}
\begin{proof}
    We first discuss the case where $B=0$.
    In this case, $(x\in X)$ is a rational singularity of type $(2,2,2,2)$ \cite[3.39.3]{Kol13} or of type $(3,3,3)$, $(2,4,4)$, $(2,3,6)$ \cite[3.39.4]{Kol13}. 
    Then the Gorenstein index belongs to $\{2,3,4,6\}$ by \cite[Chapter 3 (3.3.4)]{Flips-abundance}, and in particular, the index is prime to $p$.
    Thus, by taking $g$ as the index one cover of $(x\in X)$, we conclude that (1) holds.
    
    Next, suppose that $B$ is non-zero. In this case, 
    $(x\in X,B)$ is an lc cyclic quotient \cite[3.35 (2)]{Kol13} or a dihedral quotient \cite[3.35 (3)]{Kol13}.
    In the first case, $(x\in X,B)$ is toric by \cite[Theorem 2.13]{Lee-Nakayama}.
    In the latter case, the Cartier index of $K_X+B$ is two. This fact can be observed from the fact that the discrepancy of the exceptional leaves are $-1/2$ (see \cite[3.37]{Kol13} for the definition of leaves).
    Since $p>2$, we can take the index one cover $g\colon (x'\in X',B')\to (x\in X,B)$.
    Since $K_{X'}+B'$ is Cartier, all the discrepancy should be integer. 
    Moreover, $(X',B')$ is lc, but not plt \cite[Proposition 2.50]{Kol13}.
    Thus, $(x'\in X',B')$ have to be an lc cyclic quotient, which is toric as above. Thus, (2) holds.
\end{proof}

\begin{lem}\label{lem:reduction of F-pure lc case}
    Let $(x\in X,B)$ be an $F$-pure surface singularity such that $B$ is reduced.
    Suppose that $(X,B)$ is not plt at $x\in X$.
    Then one of the followings holds:
    \begin{enumerate}
        \item[\textup{(1)}] $B=0$ and there exists a finite Galois cover $g\colon (x'\in X')\to (x\in X)$ of degree prime to $p$ such that $X'$ is Gorenstein $F$-pure.
        \item[\textup{(2)}] $B\neq 0$ and there exists a finite Galois cover $g\colon (x'\in X',B)\to (x\in X, B)$ of degree prime to $p$ such that $(x'\in X',B)$ is a toric singularity.
    \end{enumerate}
\end{lem}
\begin{proof}
   First, we recall that $(X,B)$ is lc (Remark \ref{rem:F-pure-to-lc}). Thus, the essentially same proof as Lemma \ref{lem:reduction of lc case} works as follows:
   
   When $B=0$, we can observe that $p$ does not divide the Gorenstein index of $X$ by \cite[Theorem 1.2]{Hara(two-dim)}, and thus we can take $g$ as the index one cover of $(x\in X)$. Moreover, $F$-purity is preserved by finite covers that is \'etale in codimension one \cite[Theorem 4.8]{Hara-Watanabe}. Thus, (1) holds.
   
   When $B\neq 0$, we can reduced to the case $(x\in X,B)$ is an lc cyclic quotient since a dihedral quotient singularity is not $F$-pure if $p=2$ by \cite[Theorem 4.5 (2)]{Hara-Watanabe}. Thus, we conclude (2) by the fact that an lc cyclic quotient is toric.
\end{proof}

\subsection{The toric or Gorenstein case}

In the previous subsection, we reduced Theorem \ref{Introthm:local vanishing for F-pure} to the case where $(X,B)$ is toric or where $X$ is Gorenstein and $B=0$.
We have settled the toric case in Proposition \ref{prop:toric cases}.
In this section, we observe the Vanishing (I-1) for the latter case follows from the results by Hirokado \cite{Hirokado} and  Wahl \cite{Wahl}.

\begin{prop}\label{prop:toric}
    Let $(X,B)$ be a pair of a normal surface and a reduced divisor.
    Suppose that a singularity $(x\in X,B)$ is toric.
    Then $(X,B)$ satisfies Vanishing (I-i) at $x\in X$ for all $i\geq 0$.
\end{prop}
\begin{proof}
    Since we may assume that $(X,B)$ is a toric pair (Remark \ref{rem:vanishing for singularities}), the assertion follows from Proposition \ref{prop:toric cases}.
\end{proof}

\begin{prop}\label{prop:canonical}
    Let $X$ be a normal surface.
    Suppose that a singularity $(x\in X)$ is $F$-pure and canonical.
    Then $X$ satisfies Vanishing (I-1) at $x\in X$.
\end{prop}
\begin{proof}
    Let $f\colon Y\to X$ be the minimal resolution at $x$ with $E\coloneqq \Exc(f)$.
    Since $(x\in X)$ is canonical and $f$ is minimal, we may assume that $\omega_Y\cong \sO_Y$.
    Then the formal duality yields that
    \[
    R^1f_{*}\Omega^1_Y(\log\,E)(-E)\cong H^1_E(T_Y(-\log\,E)(E)).
    \]
    In \cite[Theorem 1.1 (ii)]{Hirokado}, all the RDPs that satisfies $H^1_E(T_Y(-\log\,E)(E))\neq 0$ are classified, and we can confirm that they coincide with non-$F$-pure RDPs.
    Here, note that $S_Y$ in \cite{Hirokado} is $T_Y(-\log\,E)$ in our notation.
    The list of $F$-pure canonical singularities can be found in \cite[Table 1]{Kawakami-Takamatsu} for example.
\end{proof}

\begin{prop}\label{prop:Gor lc}
    Let $X$ be a normal surface.
    Suppose that a singularity $(x\in X)$ is Gorenstein lc and not canonical.
    Then $X$ satisfies Vanishing (I-1) at $x\in X$.
\end{prop}
\begin{proof}
    In this case, $(x\in X)$ is a simple elliptic singularity \cite[3.39.1]{Kol13} or a cusp singularity \cite[3.39.2]{Kol13}.
    Let $f\colon Y\to X$ be the minimal good resolution at $x$ with $E\coloneqq \Exc(f)$.
    Here, we recall that the minimal good resolution is the minimal one among log resolutions.
    Since $(x\in X)$ is simple elliptic or cusp, every coefficient of a prime $f$-exceptional divisor in $f^{*}K_X-K_Y$ is equal to one, and thus we may assume that $E=f^{*}K_X-K_Y=-K_Y$.
    Then the formal duality yields that
    \begin{align*}
        R^1f_{*}\Omega^1_Y(\log\,E)(-E)=R^1f_{*}\Omega^1_Y(\log\,E)(K_Y)\cong H^1_E(T_X(-\log\,E)).
    \end{align*}
    If $E$ is irreducible, then $E$ is an elliptic curve and the vanishing of the last cohomology follows from \cite[Proposition 2.12 (a)]{Wahl}.
     If $E$ is not irreducible, then $E$ is a cycle of smooth rational curves. Then $t_i$ in \cite[(2.13.1)]{Wahl}, which is defined as $E_i\cdot (E-E_i)$, is equal to two. Thus, we can apply \cite[Proposition 2.14]{Wahl} to obtain the vanishing.
\end{proof}

\subsection{Proof of Theorems \ref{Introthm:local vanishing for F-pure} and \ref{Introthm:tame quotient}}

In this subsection, we prove Theorems \ref{Introthm:local vanishing for F-pure} and \ref{Introthm:tame quotient}.
We first show that Vanishing (I-0) holds for an lc surface pair in all characteristics:

\begin{prop}\label{prop:i=0}
    Let $(X,B)$ be an lc surface pair.
    Then Vanishing (I-0) holds.
\end{prop}
\begin{proof}
    Since $K_X+\lfloor B\rfloor$ is $\Q$-Cartier \cite[Proposition 7.2]{Gra}, replacing $B$ with $\lfloor B\rfloor$, we may assume that $B$ is reduced.
    We fix a closed point $x\in X$.
    Let $f\colon Y\to X$ be the minimal good resolution at $x$ with $E\coloneqq \Exc(f)$ and $B_Y\coloneqq f^{-1}_{*}B+E$.
    We aim to show that $R^1f_{*}\sO_Y(-B_Y)_{x}=0$.
    
    \noindent\textbf{Case 1: The case $B=0$ and $X$ is rational at $x$.}\,\, In this case,
     we conclude as follows:
    \[
    H^0(Y, \sO_Y)\twoheadrightarrow H^0(E, \sO_{E})\cong k \to R^1f_{*}\sO_Y(-E) \to R^1f_{*}\sO_Y=0,
    \]
    where we have $H^0(E, \sO_{E})\cong k$ since $E$ is reduced and connected.
    
    \noindent\textbf{Case 2: The case $B=0$ and $X$ is not rational at $x$.}\,\,
    In this case, $(x\in X)$ is a simple elliptic or a cusp singularity.
    In particular, $-B_Y=-E=K_Y$.
    Then we obtain the desired vanishing by Grauert--Riemenschneider vanishing \cite[Theorem 10.4]{Kol13}.
    
    \noindent\textbf{Case 3: The case $B\neq0$ at $x$.}\,\,
    By classification, one of the following holds:
        \begin{enumerate}
            \item $(x\in X,B)$ is an lc cyclic quotient singularity.
            \item $(x\in X,B)$ is a plt cyclic quotient singularity.
            \item $(x\in X,B)$ is a dihedral quotient singularity.
        \end{enumerate}
    If (1) or (2) holds, then the vanishing follows from Proposition \ref{prop:toric}.
    Thus, we may assume that (3) holds.
    \begin{clm}\label{cl:rational vanishing}
        $R^1f_{*}\sO_Y(-B')=0$, where $B'\coloneqq f^{-1}_{*}B$.
    \end{clm}
    \begin{clproof}
        In this case, every $f$-exceptional divisor has coefficient bigger than zero in $f^{*}(K_X+B)-K_Y$, and
        there exist prime $f$-exceptional divisors whose coefficients in $f^{*}(K_X+B)-K_Y$ are $1/2$ (see Lemma \ref{lem:reduction of lc case}).
        Then taking $N=0$ and $\Delta_X=f^{*}(K_X+B)-K_Y-B'$ in \cite[Theorem 10.4]{Kol13}, we obtain $R^1f_{*}\sO_Y(-B')=0$.
    \end{clproof}
    Now, we deduce the desired vanishing from Claim \ref{cl:rational vanishing}.
    Since $B_Y$ is connected, we can take an irreducible component $E_1\subset E$ such that $E_1\cdot B'>0$.
    Then the exact sequence
    \[
    0\to \sO_Y(-B'-E_1)\to \sO_Y(-B') \to \sO_{E_1}(-B')\to 0,
    \]
    we obtain $R^1f_{*}\sO_Y(-B'-E_1)=0$.
    Similarly, choosing an irreducible component $E_2\subset E-E_1$ so that $E_2\cdot (B'+E_1)>0$, we obtain
    $R^1f_{*}\sO_Y(-B'-E_1-E_2)=0$.
    Repeating this procedure, we conclude that 
    \[
    R^1f_{*}\sO_Y(-B_Y)=R^1f_{*}\sO_Y(-B'-E_1-\cdots-E_n)=0,
    \]
    as desired.
\end{proof}

\begin{proof}[Proof of Theorem \ref{Introthm:local vanishing for F-pure}]
    By replacing the base field with its algebraic closure, we can work over an algebraically closed field.
    We note that log canonically and $F$-purity is preserved by the base change (Remark \ref{rem:base change of F-purity}).
    By \cite[Proposition 7.2]{Gra} and Remark \ref{rem:reduction and F-purity}, by replacing $B$ with $\lfloor B\rfloor$, we may assume that $B$ is reduced.
    The case where $i=0$ follows from Proposition \ref{prop:i=0}.
    The case where $i=2$ is nothing but Grauert--Riemenschneider vanishing \cite[Theorem 10.4]{Kol13}.
    We note that so far we have not used the condition (1) or (2) in Theorem \ref{Introthm:local vanishing for F-pure}.
    
    Now, we focus on the case where $i=1$.
    We fix a closed point $x\in X$ and show that $(X,B)$ satisfies Vanishing (I-1) at $x\in X$.
    By Proposition \ref{prop:reduction to Galois cover}, Lemmas \ref{lem:reduction of F-pure dlt case}, \ref{lem:reduction of lc case}, and \ref{lem:reduction of F-pure lc case}, and Remark \ref{rem:reduction of dlt case}, taking a suitable finite Galois cover, we may assume that one of the following holds:
    \begin{enumerate}
        \item $(x\in X,B)$ is toric.
        \item $B=0$ and $(x\in X)$ is an $F$-pure RDP.
        \item $B=0$ and $(x\in X)$ is Gorenstein lc.
    \end{enumerate}
    Then the assertion follows from Propositions \ref{prop:toric}, \ref{prop:canonical}, and \ref{prop:Gor lc}, respectively.
\end{proof}

\begin{proof}[Proof of Theorem \ref{Introthm:tame quotient}]
     In the case where (1) holds, by definition, there exists a finite Galois cover $g\colon (x'\in X')\to (x\in X)$ degree prime to $p$ from a singularity $(x'\in X')$ of a smooth surface.
    Then the assertion follows from Propositions \ref{prop:reduction to Galois cover} and \ref{prop:independence}.
    
    In the case where (2) or (3) holds, we can deduce the assertion from Theorem \ref{Introthm:local vanishing for F-pure} and an argument of \cite[Section 5]{Mustata-Olano-Popa}. 
    In fact, \cite[Proposition 2.2]{Mustata-Olano-Popa} works in all characteristics, and thus the assertion can be reduced to the surjectivity of
    $\alpha\colon \bigoplus_{i}H^0(E_i,\sO_{E_i})\to H^1(Y,\Omega_Y^1)$, which is induced by the residue exact sequence.
    Let $\beta\colon H^1(Y,\Omega_Y^1)\to \bigoplus_{i}H^1(E_i,\Omega_{E_i})\cong \bigoplus_{i}H^0(E_i,\sO_{E_i})$ the map induced by the restriction.
    Then $\beta\circ \alpha$ is given by the intersection matrix $(E_i\cdot E_j)_{i,j}$, which is an isomorphism by the assumption that the determinant is prime to $p$ (see \cite[Section 8.C]{Gra}).
    By Theorem \ref{Introthm:local vanishing for F-pure}, we have $R^1f_{*}\Omega_Y^1(\log\,E)(-E)=0$. Thus, the same argument as \cite[Section 5]{Mustata-Olano-Popa} shows that $\beta$ is an isomorphism. Therefore, $\alpha$ is an isomorphism, and we conclude (2) and (3).
\end{proof}

\section{A necessary condition of Vanishing (I-1)}

Throughout this section, we work over a perfect field of characteristic $p>0$.

\subsection{Hara's Cartier operators}\label{sebsec:Hara Cartier operator}
Throughout this subsection, we use the following convention.

\begin{conv}
Let $Y$ be a smooth variety, $E$ a reduced snc divisor on $Y$, and $D$ a $\Q$-divisor on $Y$ such that the support $\Supp(\{D\})$ of the factional part is contained in $E$.
Set $D'\coloneqq \lfloor pD\rfloor-p\lfloor D\rfloor$.
Then $D'$ is an effective Cartier divisor such that $\Supp(D')\subset E$ and every coefficient of $D'$ is less than $p$.
\end{conv}

In what follows, we summarize the Cartier operator generalized by Hara \cite[Section 3]{Hara98}.
Let $j'\colon Y\setminus E\hookrightarrow Y$ be the inclusion.
Considering $\Omega^{i}_Y(\log\,E)(D')$ as a subsheaf of $(j')_{*}\Omega^{i}_{Y\setminus E}$, we obtain the complex
\[
\Omega_Y^{\mydot}(\log\,E)(D')\colon \sO_Y(D')\xrightarrow{d}\Omega^1_Y(\log\,E)(D') \xrightarrow{d} \cdots.
\]
Taking the Frobenius pushforward and tensoring with $\sO_Y(D)=\sO_Y(\lfloor D \rfloor)$, we obtain a complex
\[
F_{*}\Omega_Y^{\mydot}(\log\,B)(pD) \colon F_{*}\sO_Y(pD) \xrightarrow{F_{*}d\otimes\sO_Y(D)}
F_{*}\Omega^1_Y(\log\,E)(pD) \xrightarrow{F_{*}d\otimes\sO_Y(D)} \cdots
\]
of $\sO_Y$-modules, where 
\[
F_{*}\Omega^{i}_Y(\log\,E)(pD)\coloneqq F_{*}(\Omega^{i}_Y(\log\,E)\otimes \sO_Y(\lfloor pD\rfloor ))
\]
for all $i\geq 0$. We define coherent $\sO_Y$-modules by 
\begin{align*}
    &B^{i}_Y((\log\,E)(pD))\coloneqq \mathrm{Im}(F_{*}d\otimes\sO_Y(D) \colon F_{*}\Omega^{i-1}_Y(\log\,E)(pD) \to F_{*}\Omega^{i}_Y(\log\,E)(pD)),\\
    &Z^{i}_Y((\log\,E)(pD))\coloneqq \mathrm{Ker}(F_{*}d\otimes\sO_Y(D) \colon F_{*}\Omega^{i}_Y(\log\,E)(pD) \to F_{*}\Omega^{i+1}_Y(\log\,E)(pD)),
\end{align*}
for all $i\geq 0$.

\begin{lem}\label{lem:Cartier isomorphism (log smooth case)}
There exist the following exact sequences:
\begin{equation}
    0 \to Z^{i}_Y((\log\,E)(pD)) \to F_{*}\Omega^{i}_Y(\log\,E)(pD) \to B^{i+1}_Y((\log\,E)(pD))\to 0,\label{log smooth 1}
\end{equation}
\begin{equation}
    0 \to B^{i}_Y((\log\,E)(pD))\to Z^{i}_Y((\log\,E)(pD))\xrightarrow{C^{i}_{Y,E}(D)} \Omega^{i}_Y(\log\,E)(D)\to 0,\label{log smooth 2}
\end{equation}
for all $i\geq0$.
The map $C^{i}_{Y,E}(D)$ coincides with $C^{i}_{Y,E}\otimes \sO_Y(D)$ on $Y\setminus \Supp(D')$, where $C^{i}_{Y,E}$ is the usual logarithmic Cartier operator \cite[Theorem 7.2]{Kat70}.

Moreover, $B^{i}_Y((\log\,E)(pD))$ and $Z^{i}_Y((\log\,E)(pD))$ are locally free $\sO_Y$-modules. 
\end{lem}
\begin{proof}
See \cite[Lemma 3.2]{Kaw4}.
\end{proof}

\begin{rem}\label{rem:tensoring Z-divs}
If $D$ is a Cartier divisor, then $\Supp(D')=\emptyset$ and 
\begin{align*}
    &B^{i}_Y((\log\,E)(pD))= B^{i}_Y(\log\,E)\otimes \sO_Y(D)\\
    &Z^{i}_Y((\log\,E)(pD))= Z^{i}_Y(\log\,E)\otimes \sO_Y(D),\,\,\text{and}\\
    &C^{i}_{Y,E}(D)=C^{i}_{Y,E}\otimes \sO_Y(D)
\end{align*}
hold for all $i\geq 0$ by Lemma \ref{lem:Cartier isomorphism (log smooth case)}.
\end{rem}

\subsection{Proof of Theorem \ref{Introthm:Local vanishing to F-inj}}

In this subsection, we prove Theorem \ref{Introthm:Local vanishing to F-inj} using Hara's Cartier operator in the previous subsection.

\begin{defn}\label{def:F-inj}
Let $X$ be a variety and $x\in X$ a point. We say that $X$ is \textit{$F$-injective at $x$} if the map
\[
H^i_{\m_x}(\sO_{X, x})\xrightarrow{F} H^i_{\m_x}(\sO_{X, x})
\] 
induced by the Frobenius morphism is injective for all $i\geq 0$.
We say that $X$ is \textit{$F$-injective} if $X$ is $F$-injective at every point of $X$.
\end{defn}
\begin{rem}\label{rem:F-inj}
    If $X$ is Cohen-Macaulay, then $X$ is $F$-injective if and only if the trace map $F_{*}\omega_X\to \omega_X$ of the Frobenius map is surjective \cite[Proposition 3.19 (1)]{Takagi-Watanabe}.
\end{rem}

\begin{proof}[Proof of Theorem \ref{Introthm:Local vanishing to F-inj}]
By taking $D=-1/pE$ in \eqref{log smooth 1}, we have an exact sequence
\[
0\to Z_Y^1((\log\,E))(-E))\to F_{*}\Omega_Y^1(\log\,E))(-E) \to B^2_Y((\log\,E)(-E))\to 0
\]
We have $R^1f_{*}\Omega_Y^1(\log\,E)(-E)=0$ by assumption, and $R^2f_{*}Z_Y^1((\log\,E))(-E))=0$ since $\dim\,E\leq 1$.
Thus, we obtain \[R^1f_{*}B^2_Y((\log\,E)(-E))=0\]
by the above exact sequence.

By taking $D=-1/pE$ in \eqref{log smooth 2}, we have an exact sequence
\[
0\to B_Y^2((\log\,E))(-E))\to F_{*}\Omega_Y^2(\log\,E)(-E) \to \Omega^2_Y(\log\,E)(-1/pE)\to 0.
\]
We have $F_{*}\Omega_Y^2(\log\,E)(-E)=F_{*}\omega_Y$ and \[\Omega^2_Y(\log\,E)(-1/pE)=\omega_Y(E+\lfloor -1/pE\rfloor)=\omega_Y.\]
Since $X$ is rational, we have $f_{*}\omega_Y=\omega_X$.
Thus, taking the pushforward by $f$, we obtain an exact sequence
\[
F_{*}\omega_X\xrightarrow{\mathrm{Tr_F}}\omega_X\to R^1f_{*}B^2_Y((\log\,E)(-E))=0.
\]
Therefore, $X$ is $F$-injective (Remark \ref{rem:F-inj}). 
\end{proof}

\begin{rem}
    In the above proof, we showed that $B^2_Y((\log\,E)(-E))$ coincides with the kernel of the trace map $F_{*}\omega_Y\to \omega_Y$, which is equal to $B_Y^2$ by definition.
    In general, for a pair $(Y,E)$ of a smooth variety $Y$ and a reduced snc divisor $E$, we have $B_Y^{\dim\,Y}=B^{\dim\,Y}_Y((\log\,E)(-E))$.
\end{rem}

\section*{Acknowledgements}
The author expresses his gratitude to Teppei Takamatsu and Shou Yoshikawa for valuable conversations, and to the anonymous referee for comments that improved the exposition of this paper.
This work was supported by JSPS KAKENHI Grant number JP22KJ1771 and JP24K16897.


\end{document}